\documentclass[12pt]{article}

\usepackage[all]{xypic}
\usepackage{amsmath}
\usepackage{amsfonts}
\usepackage{amssymb}
\usepackage{amsthm}

\theoremstyle{plain}
\newtheorem{thm}{Theorem}[section]
\newtheorem{mthm}{Theorem}

\newtheorem{crl}[thm]{Corollary}
\newtheorem{mcrl}[mthm]{Corollary}

\newtheorem{prp}[thm]{Proposition}

\newtheorem{lmm}[thm]{Lemma}
\newtheorem{mlmm}[mthm]{Lemma}

\theoremstyle{remark}
\newtheorem{exm}[thm]{Example}

\newtheorem{rmk}[thm]{Remark}
\newtheorem{mrmk}[mthm]{Remark}

\newtheorem{mnote}[mthm]{Note}

\newcommand {\tb}{\textbf}
\newcommand {\mb}{\mathbb}
\newcommand {\Z}{\mb Z}

\newcommand {\C}{\mb C}

\newcommand {\im}{\textrm{im}}
\newcommand {\colim}{\textrm{colim}\ }

\newcommand {\ex}{\mathrm{excess}}

\begin{document}

\title{On the form of potential spherical classes in $H_*Q_0S^0$}

\author{Peter J. Eccles\footnote{\textit{peter.eccles@manchester.ac.uk}}, Hadi Zare\footnote{\textit{hzare@maths.manchester.ac.uk}}}
\date{}

\maketitle

\abstract{This note is about spherical classes in $H_*Q_0S^0$. A
conjecture, due to Ed. Curtis, predicts that only Hopf invariant one
and Kervaire invariant one elements will give rise to spherical
classes in $H_*Q_0S^0$. Yet, there has been no proof of this
conjecture around. Assuming that this conjecture fails, there must
exist some other spherical classes in $H_*Q_0S^0$. This note
determines the form of these potential spherical classes, and sets
the target for someone who wishes to prove the conjecture, in the
sense that correctness of the Curtis conjecture will be the same as
failure of any classes predicted in this paper being spherical.}

\tableofcontents

\section{Introduction and statement of results}

We start by considering the problem of calculating the image of the
Hurewicz homomorphism
$$h:{_2\pi_*}QX\to H_*QX$$
for ant path connected $CW$-complex $X$ of finite type, where $QX$
is given by $QX=\colim \Omega^i\Sigma^iX$ and $H_*$ denotes
$H_*(-;\Z/2)$. We then restrict our attention to the cases $X=S^n$
with $n>0$. This in return turns out to be fruitful, and we derive
some interesting results about spherical classes in $H_*Q_0S^0$
where $Q_0S^0$ denotes the base point component of $QS^0=\colim
\Omega^iS^i$. Considering the collection of spaces
$\{QS^n:n\geqslant 0\}$ and their homology enables us to determine
potential form of a spherical class in $H_*Q_0S^0$ whereas
considering the collection of spaces $\{Q_0S^{-n}:n\geqslant 0\}$
helps to eliminates some of these potential classes from being
spherical in $H_*Q_0S^0$. Here $QS^{-n}=\Omega Q_0S^{-n+1}$ with
$Q_0S^{-n}$ standing for the base
component of $QS^{-n}$.\\
We start by explaining our approach, and state our results together
with some notes and explanations comparing our results to some other results
that previously have been known.\\
Recall that given any space $Y$, spherical classes in $H_nY$ are
those element which belong to the image of the Hurewicz
homomorphism. This the makes it straightforward to see that if $y\in
H_*Y$ is spherical, then it has two basic properties:
\begin{itemize}
\item $y$ is primitive;
\item $y$ is $A$-annihilated.
\end{itemize}
Here primitive is understood to be primitive with respect to the
co-product induced by the diagonal map $Y\to Y\times Y$. By $y\in
H_*Y$ being $A$-annihilated we mean that
$$\begin{array}{ll}
Sq^i_*y=0 & \textrm{for any }i>0
\end{array}$$
where $Sq^i_*:H_*Y\to H_{*-i}Y$ is the dual to the $i$-th Steenrod
operation $Sq^i:H^*Y\to H^{*+i}Y$. One notes that not every class in
$H_*Y$ may have both properties. Hence the above properties give an
upper bound on the set of all spherical classes in $H_*Y$, although
they do not in general characterise such classes.\\
Our first result describes all $A$-annihilated classes in $H_*QX$
which are of the form $Q^Ix$ for some $x\in H_*X$. We need the
following definition to state this result. Let $n$ be a positive
integer with $n=\sum n_i2^i$ with $n_i\in\{0,1\}$. Define the
function $\rho:\mathbb{N}\to\mathbb{N}$ by
$\rho(n)=\min\{i:n_i=0\}$. We then have the following.

\begin{mthm}
Suppose $Q^Ix\in H_*QX$ is given, with $I=(i_1,\ldots,i_r)$
admissible, and $\ex(Q^Ix)>0$. Such a class is $A$-annihilated if
and only if the following conditions are satisfied:\\
$1$- $x\in\overline{H}_*X$ is $A$-annihilated;\\
$2$- $\ex(Q^Ix)<2^{\rho(i_1)}$;\\
$3$- $0\leqslant 2i_{j+1}-i_j<2^{\rho(i_{j+1})}$;\\
where $\ex(Q^Ix)=i_1-(i_2+\cdots+i_r+\dim x)$. If $l(I)=1$, then the
first two conditions chracterise all $A$-annihilated class of the
form $Q^ix$ of positive excess. Notice that the fact that
$\ex(Q^Ix)>0$ means that $Q^Ix$ is not a square, i.e. it is an
indecomposable.\\
Here $\overline{H}_*$ stands for the reduced homology;
$I=(i_1,\ldots,i_r)$ is called admissible if $i_j\leqslant 2i_{j+1}$
for all $1\leqslant j\leqslant r-1$; and the length of $I$ is
defined by $l(I)=r$.
\end{mthm}

\begin{mrmk}
Notice that if $Q^Ix$ is an $A$-annihilated class with $Q^Ix$ being
of positive excess, then $I$ cannot have any even entry. This is
easy to see, once we observe that having $i_1$ even implies that
$\rho(i_1)=0$. This together with condition $2$ of the above theorem
implies that $\ex(Q^Ix)<2^0=1$, i.e. $\ex(Q^Ix)=0$ which is a
contradiction. Moreover, if there exists $j>1$ with $i_j$ even, then
condition $3$ implies that $i_{j-1}$ is even. Iterated application
of this will imply that $i_1$ is even which leads to a
contradiction.
\end{mrmk}

Conditions $(2)-(3)$ were used in Curtis's work \cite[Theorem
6.3]{3} to describe the $0$-line of the $E_2$-term of the unstable
Adams spectral sequence converging to the $2$-primary component of
$\pi_*\Omega^nS^{n+k}$ were his condition (1), corresponding to our
condition (2), is adapted to work for any space $\Omega^nS^{n+k}$.
Curtis claims that these conditions describe a basis for the
$0$-line of the $E_2$-term of the unstable Adams spectral sequence,
whereas Wellington \cite[Remark 11.26]{15} shows that this claim is
not valid for the case $k=0$. In fact Curtis's claim is correct in odd
dimensions, and can fail only in even dimensions. This later is a
rather nontrivial corollary of Theorem 1 which is an outcome of its
proof.

\begin{mcrl}
Let $x\in\overline{H}_*X$ be fixed. Suppose
$$\xi=\sum Q^Ix$$
is $A$-annihilated, with $\ex(Q^Ix)>0$, and $I$ runs over certain
admissible sequences. Then each term $Q^Ix$ is $A$-annihilated. In
particular any odd dimensional class of the above form has this
property.
\end{mcrl}

\begin{mnote}
We note that the above corollary reflect an important fact about the
Dyer-Lashof algebra described as following. Assume that $Q^I+Q^J\in
R$ is $A$-annihilated, where $R$ is the Dyer-Lashof algebra. Assume
that $l(I)=l(J)$, $\ex(Q^I)>0$ and $\ex(Q^J)>0$. Then $Q^I$ and
$Q^J$ both are $A$-annihilated. This claim is not true if we remove
the condition on the excess, i.e. there are counter examples if we
allow some terms of trivial excess. In the next section, we will
analyse an example due to Wellington \cite[Remark 11.26]{15}. This
fixes previous theorem of Curtis \cite[Theorem 6.3]{3}.
\end{mnote}
A class of important examples satisfying the conditions of Corollary
3 is provided by cases $X=S^n$ with $n>0$. Notice that $H_*S^n$ is
given by $E_{\Z/2}(g_n)$, the exterior algebra on on generator
$g_n\in H_nS^n$. Then we have
$$H_*QS^n\simeq\Z/2[Q^Ig_n:I\textrm{ admissible },\ex(Q^Ig_n)>0].$$
Now suppose $\xi_n\in H_*QS^n$ is an odd dimensional $A$-annihilated
class, e.g. an odd dimensional spherical class spherical. We then
may write
$$\xi_n=\sum Q^Ig_n$$
with $\ex(Q^Ig_n)>0$. Combining this with Corollary 3 we obtain the
following.
\begin{mlmm}
Let $\xi_n=\sum Q^Ig_n\in H_*QS^n$ be an odd dimensional
$A$-annihilated class. Then each $Q^Ig_n$ is $A$-annihilated, i.e.
it satisfies conditions of Theorem $1$.
\end{mlmm}

The advantage of working with spherical classes is that they do pull
back. In this case, we can obtain a much stronger result as
following.

\begin{mcrl}
Suppose $n>0$ and $\xi_n\in H_*QS^n$ is spherical given by
$$\xi_n=\sum Q^Ig_n$$
modulo decomposable terms, with $\ex(Q^Ig_n)>0$. Then each $Q^Ig_n$
is $A$-annihilated.
\end{mcrl}

The proof of is obvious when $n>1$. If $\xi_n$ is odd dimensional,
then this is just the statement of Lemma 5. If $\xi_n$ is even
dimensional, then it pulls back to an odd dimensional spherical
class $\xi_{n-1}\in H_*QS^{n-1}$. Now applying Lemma 5 to
$\xi_{n-1}$ proves the lemma (and even a bit more!). But, the proof
for the case $n=1$ depends on some observations about
$A$-annihilated primitive classes in $H_*Q_0S^0$.\\
To state our next result we need to recall some facts about
$H_*Q_0S^0$ and its submodule of primitive classes. Recall that
$\pi_0QS^0\simeq\Z$. Given $n:S^0\to QS^0$ we let $[n]\in H_0Q_nS^0$
be the image of image of $1\in\overline{H}_0S^0$, the generator of
the non-base-point component in $H_0QS^0$ under the Hurewicz map
$\pi_0QS^0\to H_0QS^0$. One then has $[n]*[m]=[n+m]$. The homology
ring $H_*Q_0S^0$ is given by \cite[Part I, Lemma 4.10]{1}
$$H_*Q_0S^0\simeq\Z/2[Q^Ix_i:\ex(Q^Ix_i)>0, (I,i)\textrm{ admissible}]$$
where $x_i=Q^i[1]*[-2]$ with $*$ being the loop sum in $H_*QS^0$.
Regarding the submodule of primitives classes in this ring, there
are different ways to describe this submodule. First, we note that
the odd dimensional class $x_{2n+1}$ gives rise to a unique
primitive class, say $p_{2n+1}$, i.e. modulo decomposable terms we
have
$$p_{2n+1}=x_{2n+1}.$$
We also may define $p_{2n}=p_n^2$. Then it is well known \cite[Page
92, First Description]{100} that any primitive class in $H_*Q_0S^0$
can be written as a linear combination of terms of the form
$Q^Ip_{2n+1}$ with $I$ admissible, and not necessarily $(I,2n+1)$.
Second, we note that $x_1$ is primitive. It is well known that
applying Kudo-Araki operations $Q^i$ to any primitive class, results
in another primitive class. In particular, we obtain primitive
classes $Q^{2n}x_1$. Now, set
$$p_{2n+1}'=p_{2n+1}+Q^{2n}x_1.$$
Similarly, we may define $p'_{2n}=(p'_n)^2$. From the above
description it is obvious that any primitive class in $H_*Q_0S^0$
can be written in terms of $Q^Ip'_{2n+1}$ with $I$ admissible, and
not necessarily $(I,2n+1)$. This is the basis described by Madsen
\cite[Proposition 6.7]{8}. This is of course our favorite basis, as
$x_{2i+1}$ and $p'_{2i+1}$ show similar behavior under the action of
the Steenrod algebra. Regarding the problem of spherical classes in
$H_*Q_0S^0$ it is well known that there is Hopf invariant one
element in ${_2\pi_*}Q_0S^0$ if and only if $p'_{2i+1}$ is spherical
which will happen only if $2n+1=2^s-1$ for some $s>0$. Moreover, it
is well known there is Kervaire invariant one element in
${_2\pi_*}Q_0S^0$ if and only if $Q^{2n+1}p'_{2n+1}=(p'_{2n+1})^2$
is spherical which can happen only if $2n+1=2^s-1$ for some $s>0$
\cite[Proposition 7.3]{8}. This latter result has also been proved
by others in various equivalent form, see for example
\cite[Proposition 4.1]{4}, \cite[Theorem A]{14}. It is also known
that $Q^kp'_{2n+1}$ can not be spherical if $k\neq 2n+1$
\cite[Chapter 2, Remark 5]{100}. Our next result, which is the main
goal of this paper settles down the rest of potential spherical
classes in $H_*Q_0S^0$ and describes their form.\\ \\
\textbf{Main Theorem.} \textit{Let $\theta\in H_*Q_0S^0$ be a
spherical class which is not a Hopf invariant one class, neither a
Kervaire invariant one class. Then
$\theta$ satisfies one of the the following cases.\\
$1$- If $\sigma_*\theta\neq 0$ and $\theta$ is an odd dimensional
class, then
$$\theta=\sum Q^Ip'_{2i+1},$$
with $l(I)>1$ such that each of terms $Q^Ip'_{2i+1}$ in the above
sum is $A$-annihilated.\\
$2$- If $\sigma_*\theta\neq 0$ and $\theta$ is an even dimensional
class, then
$$\theta=\sum Q^Ip'_{2i+1}+P^2,$$
with $l(I)>1$ where $I$ has only has odd entries. In this case
$(I,2i+1)$ satisfies condition $3$ of Theorem $2$, i.e.
$0<2i_{j+1}-i_j<2^{\rho(i_{j+1})}$ for $1\leqslant j\leqslant r$
with $i_{r+1}=2i+1$. Moreover, $\ex(Q^Ip'_{2i+1})-1<2^{\rho(i_1)}$
for every $Q^Ip'_{2i+1}$ involved in the above sum. Here $P$ is a
primitive term. If $P\neq 0$, then it is of odd dimension. If $P=0$,
then each term in the above expression for $\theta$ is
$A$-annihilated.\\
$3$- If $\sigma_*\theta=0$, then $\theta=\xi^2$, with $\xi$ an odd
dimensional $A$-annihilated primitive class, i.e.
$$\theta=(\sum Q^Ip'_{2i+1})^2,$$
with $l(I)>0$ such that each of terms $Q^Ip'_{2i+1}$ in the above
sum is $A$-annihilated.\\
Moreover, assume that $f\in{_2\pi_*}Q_0S^0$ be a class with
$hf=\theta$. Then $f$ maps nontrivially under the projection
$${_2\pi_*}Q_0S^0\rightarrow{_2\pi_*}\mathrm{Coker }J$$
where $J$ is the $J$-homomorphism.\\
In all of the above cases $(I,2i+1)$ is supposed to be
admissible.}\\

Along the line, we establish a natation for $H_*Q_0S^{-1}$ which
previous was known to be an exterior algebra \cite[Theorem 1.1]{2}.
This will involve an observation on the image of homology of the
complex transfer, viewed as a map
$$Q\Sigma\C P_+\to Q_0S^0.$$
Moreover, a calculation for the submodule of primitive classes in
$H_*Q\C P$ will be presented, and will be applied in proving a part
of our main theorem.\\
This paper is organised following. We prove Theorem 1, and
corollaries implies by it, in Section 1. We then move to prove our
main theorem. We prove our main theorem in several steps, as the
proof is quite long. We have use single numbering in this section,
where as in other sections, the results are labeled based on the
number of section.\\
\tb{Acknowledgement.} This paper reports some of results announced in second named authors' PhD thesis \cite{100}, supervised by the 
first named author at the University of Manchester. The second named author is grateful to his supervisor for his support during this time. He 
and  wishes to express deep gratitude towards the school of Mathematics, and his own family
for the support during his PhD.

\section{Proof of Theorem 1}

We start by recall some fact about $H_*QX$ where $X$ is an arbitrary
path connected space. Recall that having fixed an additive basis
$\{x_\alpha\}$ for $\overline{H}_*X$, with $X$ path connected, then
$H_*QX$ is a polynomial algebra with generators given by the symbols
$Q^Ix_\alpha$, with $I$ admissible. Allowing the empty sequence
$\phi$ to be an admissible sequence, with $Q^\phi x=x$, and
$\ex(\phi)=+\infty$, then we can see that $Q^Ix$ is a decomposable
if and only if $\ex(Q^Ix_\alpha)=0>0$ where in this case
$Q^Ix=(Q^{i_2}\cdots Q^{i_r}x)^2$ with $I(i_1,\ldots, i_r)$. Hence,
Theorem 2 determines all $A$-annihilated classes in $H_*QX$ classes
of the form $Q^Ix$ with are not square. Notice that in general, any
class in $H_*QX$ involving at least one term $Q^Ix$ of positive
excess with $I$ determines a nonzero class in $QH_*QX$, the module
of indecomposables of $H_*QX$, i.e. it gives rise to an
indecomposable element.

We only use the \textit{Nishida relations}. The Nishida relation is
given as following \cite[Part I, Theorem 1.1(9)]{1},
\begin{equation}
\begin{array}{lll}
Sq^a_*Q^b&=&\sum_{r\geqslant 0}{b-a\choose a-2r}Q^{b-a+r}Sq^r_*.
\end{array}
\end{equation}

Notice that $Sq^r_*Q^I$ with $l(I)>1$ may be computed by iterated
use of the Nishida relations. One observes that \textit{the Nishida
relations respect the length}, i.e. if
$$Sq^a_*Q^I=\sum Q^KSq^{a^K}_*,$$
then $l(I)=l(K)$.\\

Let $R$ denote the Dyer-Lashof algebra. Then according to
\cite[Equation 3.2]{9} the Nishida relations maybe used to define an
action $N:A\otimes R\to R$ as following
\begin{eqnarray}
N(Sq^a_*,Q^b)&=&{b-a\choose a}Q^{b-a},\\
N(Sq^a_*,Q^{i_1}\cdots Q^{i_r}) & = & \sum{i_1-a\choose
a-2t}Q^{i_1-a+t}N(Sq^t_*,Q^{i_2}\cdots Q^{i_r}).
\end{eqnarray}
In other words suppose $Sq^a_*Q^I=\sum Q^KSq^{a^K}_*$ where
$a^K\in\Z$. Then we have
\begin{equation}
\begin{array}{lll}
N(Sq^a_*,Q^I)&=&\sum_{a^K=0}Q^K.
\end{array}
\end{equation}
We note that if a sequence $I$ is admissible, then it is not clear
whether or not after applying $Sq^a_*$ we will get a sum of
admissible terms, i.e we may need to use the Adem relations to
rewrite terms in admissible form. This means that we may decide
about vanishing or non-vanishing of a homology class $Sq^a_*Q^Ix$
after rewriting it in admissible form.

\begin{exm}
Consider $Q^9Q^5g_1$ which is an admissible term. One has
$$Sq^4_*Q^9Q^5g_1=Q^7Q^3g_1,$$
where $Q^7Q^3$ is not admissible. Although it may look nontrivial,
however the Adem relation $Q^7Q^3=0$ implies that $Q^7Q^3g_1=0$.
Indeed the class $Q^9Q^5g_1$ is not $A_*$-annihilated, which can be
seen by applying $Sq^2_*$ as we have
$$Sq^2_*Q^9Q^5g_1=Q^7Q^5g_1\neq 0.$$
Notice that the right hand side of the above equation is an
admissible term.
\end{exm}

According to the above example, part of the job in distinguishing
between $A$-annihilated and not-$A$-annihilated classes $Q^Ix$ is to
choose the right operation $Sq^a_*$ in a way that the outcome is
admissible and there is no need to use the Adem relations after the
Kudo-Araki operation. The reason being that it is practically
impossible to use the Adem relations when $l(I)$ is big. The
following lemma \cite[Lemma 6.2]{3} tells us when it is not possible
to choose the ``right'' operation and provides us with the main tool
towards the proof of Theorem 1.

\begin{lmm}
Suppose $I$ is an admissible sequence such that $2i_{j+1}-i_j<
2^{\rho(i_{j+1})}$ for all $1\leqslant j\leqslant r-1$. Let
$$N(Sq^a_*,Q^I)=\sum_{K\textrm{ admissible}} Q^K.$$

Then
$$\ex(K)\leqslant \ex(I)-2^{\rho(i_1)}.$$
\end{lmm}

The above lemma can also be obtained by combining \cite[Theorem
7.11]{15}, \cite[Theorem 7.12]{15} and \cite[Lemma 12.5]{15}. Now we
are ready to prove Theorem 1. We break it into little lemmata. The
following proves Theorem 1 in one direction.

\begin{lmm}
Let $x\in H_*X$ be $A$-annihilated, and $I$ an admissible sequence
with $\ex(Q^Ix)>0$ such that\\
$1$- $\ex(Q^Ix)<2^{\rho(i_1)}$;\\
$2$- $2i_{j+1}-i_j< 2^{\rho(i_{j+1})}$ for all $1\leqslant
j\leqslant r-1$;\\
Then $Q^Ix$ is $A$-annihilated.
\end{lmm}

\begin{proof}
Let $r>0$. Then we have the following
$$\begin{array}{llllll}
Sq^a_*Q^Ix & = & \sum Q^KSq^{a^K}_*x & = & \sum_{a^K=0} Q^Kx.
           \end{array}$$
But notice that according to Lemma 2.2
$$\ex(Q^Kx)\leqslant\ex(Q^Ix)-2^{\rho(i_1)}<0.$$
Hence the above sum is trivial, and we are done.
\end{proof}

This proves the Theorem 2 in one direction. Now we have to show that
the reverse direction holds as well. That is we have to show if
either of conditions (1)-(3) of Theorem 2 does not hold then $Q^Ix$
will be not-$A$-annihilated.

\begin{rmk}
Before proceeding, we recall a basic property of the function $\rho$
defined before Theorem 2 which is as following. Notice that given a
positive integer $n$, then $\rho(n)$ is the least integer $t$ such
that
$${n-2^t\choose 2^t}\equiv 1 \textrm{ mod }2.$$
Notice that if $n=\sum n_i2^i$ and $m=\sum m_i2^i$ are given with
$n_i,m_i\in\{0,1\}$ then ${n\choose m}=1$ mod $2$ if and only if
$n_i\geqslant m_i$ for all $i$. This makes it easy to verify the
above property for $\rho$.
\end{rmk}

The next three lemmata show that if any of conditions (1), (2) or
(3) doesn't hold, then $Q^Ix$ will not be $A$-annihilated.

\begin{lmm}
Let $X$ be path connected. Suppose $I=(i_1,\ldots,i_r)$ is an
admissible sequence, such that $\ex(Q^Ix)\geqslant 2^{\rho(i_1)}$.
Then such a class is not $A$-annihilated.
\end{lmm}

\begin{proof}
This is quite straightforward. We may use $Sq^{2^\rho}_*$ with
$\rho=\rho(i_1)$, which gives
$$\begin{array}{lll}
Sq^{2^\rho}_*Q^Ix & = & Q^{i_1-2^\rho}Q^{i_2}\cdots Q^{i_r}x+O
\end{array}$$
where $O$ denotes other terms given by
$$O=\sum_{t>0}{i_1-2^\rho\choose
2^\rho-2t}Q^{i_1-2^\rho+t}Sq^t_*Q^{i_2}\cdots Q^{i_r}x.$$ Notice
that $\ex(Q^Ix)\geqslant 2^{\rho(i_1)}$ ensures that $i_1$ is not of
the form $2^\rho$. Looking at the binary expression implies that all
coefficients in $O$ are nontrivial, and $O$ will depend on the
action of $Sq^t_*$ on terms $Q^{i_2}\cdots Q^{i_r}x$. However, all
of these terms are terms of lower excess, and they will not cancel
the first term in in the right hand side of the above relation.\\
Notice that at the right hand side of the term
$Q^{i_1-2^\rho}Q^{i_2}\cdots Q^{i_r}x$ is obviously admissible.
Moreover,
$$\ex(Q^{i_1-2^\rho}Q^{i_2}\cdots Q^{i_r}x)=\ex(Q^Ix)-2^\rho\geqslant 0.$$
This proves that $Sq^{2^\rho}_*Q^Ix\neq 0$. Notice that if
$\ex(Sq^{2^\rho}_*Q^Ix)=0$, then
$$Sq^{2^\rho}_*Q^Ix  = (Q^{i_2}\cdots
Q^{i_r}x)^2\neq 0.$$ This completes the proof.
\end{proof}

The above lemma shows that if (2) of Theorem 2 does not hold, then
we will have a class which is not $A$-annihilated. Next, we move on
to the case when condition (3) does not hold.

\begin{lmm}
Let $X$ be path connected. Suppose $I=(i_1,\ldots,i_r)$ is an
admissible sequence, and let $Q^Ix$ be given with $\ex(Q^Ix)>0$ such
that $2i_{j+1}-i_j\geqslant 2^{\rho(i_{j+1})}$ for some $1\leqslant
j\leqslant r-1$. Then such a class is not $A$-annihilated.
\end{lmm}

\begin{proof}
Assume that $Q^Ix$ satisfies the condition above. We may write this
condition as
$$i_j-2^\rho\leqslant 2i_{j+1}-2^{\rho+1}=2(i_{j+1}-2^\rho),$$
where $\rho=\rho(i_{j+1})$. This is the same as admissibility
condition for the pair $(i_j-2^\rho,i_{j+1}-2^\rho)$. In this case
we use $Sq^{2^{\rho+j}}_*$ where we get
$$\begin{array}{lll}
Sq^{2^{\rho+j}}_*Q^Ix&=&Q^{i_1-2^{\rho+j-1}}Q^{i_2-2^{\rho+j-2}}\cdots
Q^{i_j-2^\rho}Q^{i_{j+1}-2^\rho}Q^{i_{j+2}}\cdots Q^{i_r}x+O
\end{array}$$
where $O$ denotes other terms, and similar to previous lemma will be
a sum of terms of lower excess. The first term in right hand side of
the of the above equality is admissible. Moreover,
$$\begin{array}{lll}
\ex(Sq^{2^{\rho+j}}_*Q^Ix) & = & (i_1-2^{\rho+j-1})-(i_2-2^{\rho+j-2})-(i_j-2^\rho)-(i_{j+1}-2^\rho)-\\
                           &   & (i_{j+2}+\cdots+i_r+\dim x)\\
                           & = & i_1-(i_2+\cdots+i_r+\dim x)\\
                           & = & \ex(Q^Ix)>0,
\end{array}$$
where by abuse of notation we have written
$\ex(Sq^{2^{\rho+j}}_*Q^Ix)$ to denote the excess of the first term
in the above equality. This implies that
$$Sq^{2^{\rho+j}}_*Q^Ix\neq 0,$$
and hence completes the proof.
\end{proof}

\begin{rmk}
According to the proof in this case we always have
$$\ex(Sq^{2^{\rho+j}}_*Q^Ix)>0$$
which means that we always end up with an indecomposable term after
applying the ``right'' operation, i.e. the outcome will not be a
square. This little observation will be useful.
\end{rmk}

Now we show that the condition (1) is also necessary in the proof of
the main theorem.

\begin{lmm}
Let $X$ be path connected, and $x\in \overline{H}_*X$ be not
$A$-annihilated. Then $Q^Ix$ is not $A$-annihilated.
\end{lmm}

\begin{proof}
Let $t$ be the least number that
$$Sq^{2^t}_*x\neq 0.$$
If $I=(i_1,\ldots,i_r)$, we apply $Sq^{2^{t+r}}_*$ to $Q^Ix$, where
we get
$$\begin{array}{lll}
Sq^{2^{r+t}}_*Q^Ix & = & Q^{i_1-2^{r+t-1}}\cdots
Q^{i_r-2^t}Sq^{2^t}_*x+O,\end{array}$$ where $O$ denotes sum of
other terms which are of the form $Q^Jy$ with $\dim y>\dim
Sq^{2^t}_*x$. This means that the first term in the above equality
will not cancel with any of other terms.\\
By abuse of notation we write $\ex(Q^{i_1-2^{r+t-1}}\cdots
Q^{i_r-2^t}Sq^{2^t}_*x)$ to denote the excess of the first term in
the above equality. We have $\ex(Q^{i_1-2^{r+t-1}}\cdots
Q^{i_r-2^t}Sq^{2^t}_*x)=\ex(Q^Ix)>0$. Moreover,
$$Q^{i_1-2^{r+t-1}}\cdots Q^{i_r-2^t}$$
is admissible. Hence $Sq^{2^{t+r}}_*Q^Ix\neq 0$. Note that similar
to the previous lemma we end up with an indecomposable term.
\end{proof}

This completes the final step in the proof of Theorem 1. Our next
task is to prove Corollary $3$, which is very important for us. We
recall that according to Corollary $3$ if we have an $A$-annihilated
sum of terms of the form $Q^Ix$ with $\ex(Q^Ix)>0$, then each of
these terms must be $A$-annihilated.

\subsection{Separating not-$A$-annihilated classes}
Here we give a sketch of the proof for Corollary 3, and refer the
reader to \cite[Subsection 3.2]{100} for more details. We show that
if $Q^Ix$ and $Q^Jx$ are two terms which are not $A$-annihilated,
both of positive excess, then their sum is not $A$-annihilated. This
then will prove Corollary $3$, as well as the general claim Note 4.\\

Theorem 1 provides us with a complete description of $A$-annihilated
classes in $H_*QX$ of the form $Q^Ix$ of positive excess, with $X$
being path connected. In return, this also classifies all such
classes that are not $A$-annihilated. For instance given a class
$Q^Ix\in H_*QX$, this class will not be $A$-annihilated, if at least
one of conditions in Theorem 2 does not hold; i.e.
\begin{itemize}
\item $x$ is not $A$-annihilated,
\item $\ex(Q^Ix)\geqslant 2^{\rho(i_1)}$;
\item There exists $j$ such that $2i_{j+1}-i_j\geqslant 2^{\rho(i_{j+1})}$.
\end{itemize}
Therefore, proving Corollary $3$ comes down to analyse these cases
and showing that if we have two terms $Q^Ix$ and $Q^Jx$ with
$l(I)=l(J)$, both of positive excess, and each being
not-$A$-annihilated for one of the above reasons the $Q^Ix+Q^Jx$ is
not $A$-annihilated as well. It is obvious, that if the two terms
are not $A$-annihilated for the same reason, then their sum is not
$A$-annihilated. For instance we have the following example.

\begin{lmm}
Suppose $Q^Ix, Q^Jx\in H_*QX$ are given, $l(I)=l(J)=r$, with
$x\in\overline{H}_*X$ not being $A$-annihilated. Then $Q^Ix+Q^Jx$ is
not $A$-annihilated.
\end{lmm}

\begin{proof}
We need to find an integer $k$ such that $Sq^k_*Q^Ix\neq
Sq^k_*Q^Jx$. We do the same as we did in the proof of Lemma 2.8. Let
$t$ be the least number such that $Sq^{2^t}_*x\neq 0$. We then
observe that
$$Sq^{2^{t+r}}_*Q^Ix\neq Sq^{2^{t+r}}_*Q^Jx.$$
This completes the proof.
\end{proof}
As another example consider the following case.
\begin{lmm}
Let $Q^Ix$ and $Q^Jx$ be two classes that are not $A$-annihilated
such that $\ex(Q^Ix)\geqslant 2^{\rho(i_1)}$, and
$\ex(Q^Jx)\geqslant 2^{\rho(j_1)}$. Then there exists $k$ such that
$$Sq^k_*Q^I\neq Sq^k_*Q^Jx.$$
\end{lmm}

\begin{proof}
If we choose $\rho=\min\{\rho(i_1),\rho(j_1)\}$, then it is clear
that
$$Sq^{2^\rho}_*Q^Ix\neq Sq^{2^\rho}_*Q^Jx.$$
\end{proof}

The key ingredient is that if $Sq^k_*Q^Ix\neq 0$ with $k$ being
least such number, and $Sq^{k'}_*Q^Jx\neq 0$ with $k'$ being least
such number. We then may use $Sq^{\min{k,k'}}_*$ to show that
$Q^Ix+Q^Jx$ is not $A$-annihilated. We refer the reader to
\cite[Subsection 3.2]{100} for more details.

Finally, we give an example showing that the conditions
$\ex(Q^I)>0$, $\ex(Q^J)>0$ are necessary. The following example is
due to Wellinton \cite[Remark 11.26]{15}.

\begin{exm}
One may check that the following class is an $A$-annihilated element
in the Dyer-Lashof algebra $R$,
$$(Q^{2062}Q^{1031}Q^{519}Q^{263}Q^{135}Q^{71}Q^{39})^2+Q^{4120}Q^{2062}Q^{1031}Q^{519}Q^{263}Q^{135}Q^{71}Q^{39}$$
Notice that according to Theorem 1
$Q^{1031}Q^{519}Q^{263}Q^{135}Q^{71}Q^{39}$ which is involved in
both terms is $A$-annihilated. This then makes it easy to see that
the above sum is $A$-annihilated, as there are a few operations
available at this stage which may map the above sum nontrivially.
However, each term is not $A$-annihilated under the action of
$Sq^2_*$. Notice that the first term is a square, and of trivial
excess, and so does not satisfy conditions of Corollary 3.
\end{exm}

Lemma 5 is just a special case of Corollary 3, and Corollary 6 is
implied by Corollary 3 and Lemma 5. Recall that according to
Corollary 6 if $\xi_n=\sum Q^Ig_n\in H_*QS^n$ is spherical, then
every term in the above sum is $A$-annihilated. This is the
statement of Lemma 5 if $\xi_n$ is odd dimensional. Let $n>1$, and
let $\xi_n$ be even dimensional. As $\xi_n$ is spherical, then it
pulls back to an odd dimensional spherical class $\xi_{n-1}\in
H_*QS^{n-1}$. According to Lemma 5
$$\xi_{n-1}=\sum Q^Ig_{n-1}$$
with each $Q^Ig_{n-1}$ being $A$-annihilated. Applying suspension to
the above class we obtain,
$$\xi_n=\sum Q^Ig_n$$
with each $Q^Ig_n$ being $A$-annihilated. This completes the proof
of Corollary 6.

\section{Proof of The Main Theorem}
Our main theorem is a combination of Corollary 3, and facts implies
by it in Remark 4, together with a comparison between different
bases for the submodule of primitives in $H_*Q_0S^0$. As stated, our
theorem result, couples this results with the behavior of the
potential spherical classes $\theta\in H_*Q_0S^0$ under the homology
suspension
$$\sigma_*:H_*Q_0S^0\to H_{*+1}QS^1.$$
If a spherical class does survive under the suspension homomorphism,
then our main theorem reduces to a form of Remark 4. But if our
spherical class $\theta\in H_*Q_0S^0$ dies under the homology
suspension, it then tells us that $\theta$ is a decomposable
primitive, i.e. a square. This then implies that
$\theta=\zeta^{2^t}$ for some $t>0$. We shall show that it is not
possible to have a spherical class $\theta=\zeta^{2^t}\in H_*Q_0S^0$
with $t>1$. This will be achieved by a preparation on the behavior
of the $S^1$-transfer, together with some observations on
$H_*Q_0S^{-1}$, $H_*Q_0S^{-2}$ and $H_*Q\Sigma^{-1}\C P_+$. Finally,
we note that it is in general true that if $\theta=\zeta^{2^t}\in
H_*QS^n$ is spherical with $n>1$, then $t<2$. The proof of this fact
is a simplified version of its proof for the case when $n=0$ where
one will need to desuspend only once and apply $Sq^1_*$. We leave
the details to reader. A proof of this is to be found in \cite[Page
60]{100}.

\subsection{The homology rings $H_*Q_0S^{-1}$, $H_*Q\Sigma^{-1}\C P$}
We start by recalling some fact about the Eilenberg-MacLane spectral
sequence. This spectral sequence is one of the main tools in
calculating the homology of loop spaces is the Eilenberg-Moore
spectral sequence. We recall the following \cite[Proposition
7.3]{5}.

\begin{prp}
Let $X$ be simply connected, with $H^*X$ polynomial. Then $H_*\Omega
X$ is an exterior algebra, and the suspension
$$\sigma_*:QH_*\Omega X\to PH_*X$$
is an isomorphism, and the Eilenberg-Moore spectral sequence
$$E^2=\mathrm{Cotor}^{H_*X}(\Z/2,\Z/2)\Rightarrow H_*\Omega X$$
collapses. In particular,
$$H_*\Omega X\simeq E_{\Z/2}(\sigma^{-1}_*PH_*X),$$
where $E_{\Z/2}(\sigma^{-1}_*PH_*X)$ denotes the exterior algebra
over $\Z/2$ generated by $\sigma^{-1}_*PH_*X$
\end{prp}

The next result identifies some cases where $H^*Q_0X$ is a
polynomial algebra. Recall that given any space, one may define the
Frobenius homomorphism $s:H^*X\to H^*X$ as before, i.e. $s(x)=x^2$.
One then has the following \cite[Lemma 7.2]{5} .

\begin{lmm}
The cohomology algebra $H^*Q_0X$ is a polynomial algebra if
$s:H^*X\to H^*X$ is injective. Here $Q_0X$ denotes the base point
component of $QX$.
\end{lmm}

The above theorems provides the main tool to calculate the homology
rings $H_*Q_0S^{-1}$ and $H_*\Omega_0QP$, where one chooses
$X=\overline{Q_0S^0},\overline{QP}$. Here $\overline{Y}$ denotes the
universal cover of a given space $Y$. We refer the reader to
\cite{5} for the proof of the machinery provided above. We recall
the calculation of $H_*Q_0S^{-1}$.

\begin{exm}
First, notice that the squaring map $H^*S^0\to H^*S^0$ is injective.
This implies that $H^*Q_0S^0$ is polynomial. Recall from Appendix D
that $Q_0S^0=P\times\overline{Q_0S^0}$. Hence $H^*\overline{Q_0S^0}$
is polynomial as well. On the other hand notice that $QS^{-1}=\Omega
Q_0S^0$, which implies that $Q_0S^{-1}=\Omega\overline{Q_0S^0}$. Now
putting $X=\overline{Q_0S^0}$ in Proposition 5.4 implies that
$H_*Q_0S^{-1}$ is an exterior algebra, with
$\sigma_*:QH_*Q_0S^{-1}\to PH_*Q_0S^0$ an isomorphism, i.e.
$$H_*Q_0S^{-1}=E_{\Z/2}(\sigma_*^{-1}PH_*Q_0S^0).$$
This is due to Cohen-Peterson \cite[Theorem 1.1]{2}.
\end{exm}

We shall combine the information by this example with the
$S^1$-example to establish our notation for $H_*Q_0S^{-1}$ which
will provide us with a ``geometric description'' of its generators.
But, first we look into another example.

\begin{exm}
Let $X=\C P, \C P_+$. Then the Frobenius homomorphisms $H^*X\to
H^*X$ is injective. We then obtain
$$\begin{array}{lll}
H_*Q\Sigma^{-1}\C P & \simeq & E_{\Z/2}(\sigma^{-1}_*PH_*Q\C P).
\end{array}$$
\end{exm}

The above example calculates $H_*Q\Sigma^{-1}\C P$. We only need to
describe the submodule of primitives in this algebra. However,
unlike the case of $H_*Q_0S^{-1}$, we do not have a natural and
geometric way of identifying the generators of $H_*Q\Sigma^{-1}\C P$.\\

First, we deal with $H_*Q_0S^{-1}$ and determine its structure as a
module over the Steenrod algebra $A$, and the Dyer-Lashof algerba
$R$.\\
The $S^1$-transfer is a map $\lambda_\C:Q\Sigma\C P_+\to QS^0$. The
homology of this map is known based on the work of
Mann-Miller-Miller \cite[Lemma 7.4]{11}. It factors through the
complex $J$-homomorphism
$$J_\C:U\to Q_1S^0.$$
Using the translation map $*[-1]$ we then will land in $Q_0S^0$. The
map $\lambda_\C$ is an infinite loop map, obtained as the infinite
loop extension of the composite
$$\xymatrix{ \Sigma\C P_+\ar[r]   & U\ar[r]  &Q_1S^0\ar[r]  & Q_0S^0.}$$
The mapping $\lambda_\C$ may be viewed as an extension of
$\nu:S^3\not\to S^0$ where $S^3$ sits as the bottom cell of
$\Sigma\C P$, and the inclusion maps $g_3$ to $\Sigma c_2$ in
homology. Hence $\Sigma c_2$ maps to $x_3+O_3$. In fact, $\Sigma
c_2$ maps to $p'_3$ as this has to give the action of $\nu:S^3\to
Q_0S^0$. Using this, and the action of Steenrod algebra on $\Sigma\C
P$, we observe that $\Sigma c_{2i}$ maps to $x_{2i+1}+O_{2i+1}$
where $O_{2i+1}$ denotes the other terms. On the other hand, notice
that $\Sigma c_i$ is primitive. Also, the image must have the same
behavior under the action of the Steenrod algebra as $\Sigma
c_{2i}$. Hence we obtain,
$$(\lambda_\C)_*\Sigma c_{2i}=p_{2i+1}+Q^{2i}x_1=p'_{2i+1}.$$
Moreover, notice that $\Sigma c_0$ maps to $x_1=p_1=p_1'$ where
$c_0$ is the generator coming from the disjoint base point. This
then allows one to calculate $(\lambda_\C)_*:H_*Q\Sigma\C P_+\to
H_*Q_0S^0$. Notice that this in particular implies that
$(\lambda_\C)_*:PH_*Q\Sigma\C P_+\to PH_*Q_0S^0$ is an epimorphism.
We are now ready to complete the calculation of $H_*Q_0S^{-1}$.

\begin{thm}
The homology algebra $H_*Q_0S^{-1}$ as an $R$-module is given by
$$E_{\Z/2}(Q^Iw'_{2i}:I\textrm{ admissible}, \dim I>2i),$$
with $w'_{2i}=(\Omega\lambda_\C)_*c_{2i}$ which satisfies
$\sigma_*w'_{2i}=p'_{2i+1}$. Two generators $Q^Iw'_{2i}$ and
$Q^Jw'_{2j}$ may be identified if and only if they map to the same
element in $H_*Q_0S^0$ under the homology suspension
$\sigma_*:H_{*-1}Q_0S^{-1}\to H_*Q_0S^0$. The behavior of generators
$w'_{2i}$ under the Steenrod operation is very much like $c_{2i}\in
H_{2i}\C P$, i.e.
$$Sq^{2k}_*w'_{2i}={i-k\choose k}w'_{2i-2k}.$$
This together with the Nishida relations completely determines the
$A$-module structure of $H_*Q_0S^{-1}$.\\
Moreover, the the mapping
$$(\Omega\lambda_\C)_*:H_*Q\C P_+\to H_*Q_0S^{-1}$$
is an epimorphism.
\end{thm}

\begin{proof}
The fact that $(\lambda_\C)_*\Sigma c_{2i}=p'_{2i+1}$ allows us to
define unique elements $w'_{2i}\in H_{2i}Q_0S^{-1}$ by
$$(\Omega\lambda_\C)_*c_{2i}=w'_{2i}.$$
Notice that the space $Q_0S^{-1}$ is an infinite loop space, and
hence we may consider terms of the form $Q^Iw'_{2i}\in
H_*Q_0S^{-1}$. These classes have the property that
$$\sigma_*Q^Iw_{2i}'=Q^Ip_{2i+1}'$$
where $\sigma_*$ denotes the homology suspension. The fact that
elements of the form $Q^Ip'_{2i+1}$ generate all primitives in
$H_*Q_0S^0$ implies that elements of the form $Q^Iw'_{2i}$ generate
$QH_*Q_0S^{-1}$, and therefore $H_*Q_0S^{-1}$ is the exterior
algebra generated by $Q^Iw'_{2i}$ with $I$ admissible. This also
determines the action of the Dyer-Lashof algebra on the homology
ring $H_*Q_0S^{-1}$. Moreover, our definition of the generators
$w_{2i}'$ allows us to derive the action of the Steenrod operation
on these classes, namely we have
$$Sq^{2k}_*w'_{2i}={i-k\choose k}w'_{2i-2k}.$$
This together with the Nishida relations describes the action of the
Steenrod algebra on the generators $Q^Iw'_{2i}$, and hence
completely determines the action of the Steenrod algebra on the
homology ring $H_*Q_0S^{-1}$.\\
Finally notice that although we have identified generators of
$H_*Q_0S^{-1}$, however there are some relations among these
generators. For example consider $Q^3x_1=x_1^4=Q^2Q^1x_1\in
H_4Q_0S^0$. Hence in $H_*Q_0S^{-1}$ we have
$$Q^3w_0=Q^2Q^1w_0.$$
This then shows that two generators in this algebra maybe identified
if they map to the same primitive class under the homology
suspension. Finally, the definition of the generators $w'_{2i}$
shows that the looped transfer induces an epimorphism in homology.
This completes the proof.
\end{proof}

It is possible to choose a different set of generators for the
submodule of primitive in $H_*Q_0S^0$ in order to give a
presentation of $H_*Q_0S^{-1}$ with no relation among its generators
\cite[Proposition 5.37]{100}. However, this description is not as
much easy to work with as the above presentation. As the above
presentation is adequate for our purpose, we then choose to work
with this description. Indeed, it is possible to identify a
subalgebra of $H_*Q_0S^{-1}$ with no relations among its generators.
\begin{lmm}
There is no relation among the generators of the subalgebra of
$H_*Q_0S^{-1}$ given by
$$E_{\Z/2}(Q^Iw'_{2i}:(I,2i+1)\textrm{ admissible, }\ex(I,2i+1)>0),$$
with $\ex(I,2i+1)=i_1-(i_2+\cdots+i_r+2i+1)$ where
$I=(i_1,\ldots,i_r)$.
\end{lmm}
We leave the proof to the reader, as it is an obvious outcome of
applying the homology suspension.\\

Next, we determine the submodule of primitives in $H_*Q\C P$ and
$H_*Q_0S^{-1}$. This will imply that $(\Omega\lambda_\C)_*$ is an
epimorphism when restricted to the submodules. We quote the result
on this and refer the reader to \cite[Section 5.8]{100} for proofs
and details of the calculations.

\begin{prp}
Any primitive class $H_*Q_0S^{-1}$ maybe written as linear
combination of classes of the form $Q^L p_{4n+2}^{S^{-1}}$ and
$Q^Kp^{S^{-1}}_{i,j}$ where $L$ and $K$ are chosen to be admissible.
The primitive classes $p_{4n+2}^{S^{-1}}$, $Q^Kp^{S^{-1}}_{i,j}$ are
defined by
$$\begin{array}{lll}
p_{4n+2}^{S^{-1}}&=&w'_{4n+2}\in H_{4n+2}Q_0S^{-1}\\
p^{S^{-1}}_{i,j} &=& Q^{2i+1}w'_{2j}\in H_{2i+2j+1}Q_0S^{-1},
\end{array}$$
modulo decomposable terms, with $j$ being even.\\
Similarly, any primitive class in $H_*Q\C P$ maybe written as linear
combination of classes of the form $Q^L p_{4n+2}^{\C P}$ and
$Q^Kp^{\C P}_{i,j}$ where $L$ and $K$ are chosen to be admissible.
The primitive classes $p_{4n+2}^{\C P}$, $Q^Kp^{\C P}_{i,j}$ are
defined by
$$\begin{array}{lll}
p_{4n+2}^{\C P}&=& c_{4n+2}\in H_{4n+2}Q\C P\\
p^{\C P}_{i,j} &=& Q^{2i+1}c_{2j}\in H_{2i+2j+1}Q\C P,
\end{array}$$
modulo decomposable terms, with $j$ being even. Moreover,
$$\begin{array}{lll}
(\Omega \lambda_\C)_*p_{4n+2}^{\C P}&=&p_{4n+2}^{S^{-1}},\\
(\Omega \lambda_\C)_*p^{\C P}_{i,j} &=&p^{S^{-1}}_{i,j}.
\end{array}$$
\end{prp}

Recall from Example 3.4, that we may apply Proposition 3.4 to
calculate $H_*Q\Sigma^{-1}\C P$ using
$$H_*Q\Sigma^{-1}\C P  \simeq  E_{\Z/2}(\sigma^{-1}_*PH_*Q\C P),$$
and our calculation of the primitive classes. We then have the
following observation.
\begin{prp}
As an $R$-module $H_*Q\Sigma^{-1}\C P$ is given by the exterior
algebra over the generators $Q^Iv_{4n+1}^{\C P}$ and
$Q^Lv_{i,j-1}^{\C P}$ with $I$ and $L$ admissible, and $\dim I>4n+1$
and $\dim L>2i+2j$. Here
$$\begin{array}{lll}
\sigma_*v_{4n+1}^{\C P} & = & p_{4n+2}^{\C P},\\
\sigma_*v_{i,j-1}^{\C P}& = & p_{i,j}^{\C P},
\end{array}$$
with $v_{4n+1}^{\C P}\in QH_{4n+1}Q\Sigma^{-1}\C P$ and
$v_{i,j-1}^{\C P}\in QH_{2i+2j}Q\Sigma^{-1}\C P$ where $j$ is even.
The generators $Q^Iv_{4n+1}^{\C P}$ are independent from each other
for different choices of admissible $I$. Two generators of the form
$Q^Lv_{i,j-1}^{\C P}$ are identified if they map to the same class
in $H_*Q\C P$ under the homology suspension.
\end{prp}

Notice that the above presentation does not give a geometric meaning
for the generators of $H_*Q\Sigma^{-1}\C P$, i.e. we do not know of
a natural way to define the generators $v_{4n+1}^{\C P}$ and
$v_{i,j-1}^{\C P}$. However, this description is enough to for our
purpose. Notice that any spherical class $H_*Q_0S^{-1}$ is
primitive, and hence lies in the image of $(\Omega\lambda_\C)_*$.
Moreover, the description of $H_*Q\Sigma^{-1}\C P$ tell us that any
primitive class in $H_*Q\C P$ pulls back through the homology
suspension. The following observation is then clear.

\begin{lmm}
Suppose $\xi_{-1}\in H_*Q_0S^{-1}$ is a spherical class. This pulls
back to a spherical class in $\xi_{-2}\in H_*QS^{-2}$ which lives in
an exterior subalgebra of $H_*Q_0S^{-2}$ given by
$$E_{\Z/2}(Q^Iv_{4n+1},Q^Lv_{i,j-1}:I,L\textrm{ admissible}),$$
where $v_{4n+1}\in H_{4n+1}QS^{-2}$ and $v_{i,j-1}\in
H_{2i+2j}QS^{-2}$ are defined by
$$\begin{array}{lll}
v_{4n+1} &=&(\Omega^2\lambda_\C)_*v_{4n+1}^{\C P} ,\\
v_{i,j-1}&=&(\Omega^2\lambda_\C)_*v_{i,j-1}^{\C P}.
\end{array}$$
Moreover, the definition of these classes imply that
$$\begin{array}{lll}
\sigma_*v_{4n+1} &=&p_{4n+2}^{S^{-1}},\\
\sigma_*v_{i,j-1}&=&p_{i,j}^{S^{-1}}.
\end{array}$$
Kudo's transgression theorem also implies that
$$\begin{array}{lll}
\sigma_*Q^Iv_{4n+1} &=&Q^Ip_{4n+2}^{S^{-1}},\\
\sigma_*Q^Lv_{i,j-1}&=&Q^Lp_{i,j}^{S^{-1}}.
\end{array}$$
\end{lmm}

We conclude this section with the following theorem.

\begin{thm}
Let $\xi=\zeta^{2^t}\in H_*Q_0S^0$ be a spherical class with
$\sigma_*\zeta\neq 0$. Then it is impossible to have $t>1$.
\end{thm}

\begin{rmk}
We would like to recall some facts about Hopf algebras
\cite[Propositon 4.23]{12} which will be use in the proof of the
above theorem. Suppose $H$ is a connected bicommutative Hopf algebra
of finite type over $k=\Z/2$. Then there is an exact sequence of the
following form
$$0\to P(s H)\to PH\to QH\to Q(r H)\to 0.$$
Here Frobenius homomorphism $s=s_H:H\rightarrow H$ is given by
$s_H(h)=h^2$. The map $QH\to Qk(r H)$ is the square root map.\\
In particular, examples $H=H_*Q_0S^n$ with $n\in\Z$ satisfy the
conditions stated above. The square root map $r:H\to H$ in our
examples has the property that
$$\begin{array}{lll}
rQ^{2n}  &=&Q^nr,\\
rQ^{2n+1}&=&0.
\end{array}$$
The above theorem then tells us that if $\xi\in H$ is a primitive
which is not a square, then the image of the $\xi$ in $QH$ must
belong to the kernel of the square root map. For example, in
$H_*Q_0S^0$ we have that
$$\begin{array}{lll}
rx_{2i}  &=&x_i,\\
rx_{2i+1}&=&0.
\end{array}$$
Hence, if $Q^Ix_m+D$, with $D$ being a sum of decomposable terms, is
a primitive class which is not a square then $Q^Ix_m$ must belong to
the kernel of the square root map $r:QH_*Q_0S^0\to QH_*Q_0S^0$. This
then implies that either $I$ must have an odd entry, or $m$ is odd.
\end{rmk}

\begin{proof}
We do the proof for $t=2$, i.e. $\xi=\zeta^4=Q^{2d}Q^d\zeta$ where
$d=\dim\zeta$. Other cases are similar. Since $\xi$ is an
$A$-annihilated primitive, then $\zeta$ also must be an
$A$-annihilated primitive class. As we recall at the beginning, we
may write $\zeta$ as a sum, $Q^Jp'_{2j+1}$ where $J$ is admissible
but $(J,2j+1)$ is not necessarily admissible. Hence we may write
$$\xi=\sum Q^{2d}Q^d Q^Lp'_{2l+1},$$
with $L$ admissible, taking all of the above terms into one big sum,
where some $(L,2l+1)$ are admissible and some are not. Such a class
pulls back to a $4d-1$ dimensional class $\xi_{-1}\in
H_{4d-1}Q_0S^{-1}$ given by
$$\xi_{-1}=\sum Q^{2d}Q^d Q^Lw'_{2l}+D_{-1},$$
where $D_{-1}$ denotes the decomposable part. This is an odd
dimensional primitive class, i.e. its indecomposable part must
belong to the kernel of the square root map. Hence either $d$ is
odd, $L$ has at least one odd entry, or $l$ is odd. Hence we may
rewrite the above class as
$$\xi_{-1}=\sum_{l\textrm{ odd}}Q^{2d}Q^dQ^Lw'_{2l}+\sum_{l\textrm{
even}}Q^{2d}Q^dQ^Lw'_{2l}+D_{-1}.$$ We have already calculated the
set of primitive classes in $H_*QS^{-1}$, hence we write
$$\xi_{-1}=\sum_{l\textrm{ odd}}Q^{2d}Q^dQ^Lp_{2l}^{S^{-1}}+\sum_{l\textrm{
even}} Q^{2d}Q^Kp_{k,l}^{S^{-1}}+D_{-1},$$ where $D_{-1}$ is an odd
dimensional decomposable primitive class. Hence $D_{-1}=0$, i.e.
$$\xi_{-1}=\sum_{l\textrm{ odd}}Q^{2d}Q^dQ^Lp_{2l}^{S^{-1}}+\sum_{l\textrm{
even}} Q^{2d}Q^Kp_{k,l}^{S^{-1}}.$$ Notice that
$Q^dQ^Lp_{2l}^{S^{-1}}$ and $Q^Kp_{k,l}^{S^{-1}}$ are of dimension
$2d-1$. We plan make use of $Sq^1_*$, but not here as in this
exterior algebra $Sq^1_*\xi_{-1}$ will be a square which is trivial
in the exterior algebra. Instead we desuspend once more. The class
$\xi_{-1}$ pulls back to a spherical class $\xi_{-2}\in
H_*Q_0S^{-2}$. according to Lemma implies that $\xi_{-2}$ is in the
exterior subalgebra generated by $\im(\Omega^2\lambda_\C)_*$. Hence,
we may write
$$\xi_{-2}=\sum_{l\textrm{ odd}}Q^{2d}Q^dQ^Lv_{2l-1}+\sum_{l\textrm{ even}} Q^{2d}Q^Kv_{k,l-1}+D_{-2}.$$
where $D_{-2}$ denotes the decomposable part. The classes
$Q^dQ^Lv_{2l-1}$ and $Q^Kv_{k,l-1}$ are of dimension $2d-2$.
Although one may decide to rewrite this sum in terms of primitives,
however this form is enough for us to get a contradiction. Observe
that
$$Sq^1_*\xi_{-2}=\sum_{l\textrm{ odd}}Q^{2d-1}Q^dQ^Lv_{2l-1}+\sum_{l\textrm{ even}} Q^{2d-1}Q^Kv_{k,l-1}+Sq^1_*D_{-2}.$$
Notice that
$Sq^1_*D_{-2}$ is a decomposable. On the other hand terms
$Q^{2d-1}Q^dQ^Lv_{2l-1}$ and $Q^{2d-1}Q^Kv_{k,l-1}$ are separated
under this action as the map to distinct terms under the homology
suspension, which also shows that these classes do not belong to
$\ker\sigma_*$. This shows that $Sq^1_*\xi_{-2}\neq 0$. But this is
a contradiction to the fact that $\xi_{-2}$ is $A$-annihilated.
Hence we have completed the proof.
\end{proof}

A similar approach maybe taken to prove the following.
\begin{lmm}
Let $\theta=\zeta^2\in H_*Q_0S^0$ be a spherical class. Then $\zeta$
must be an odd dimensional class.
\end{lmm}

We leave the proof to the reader, and only note that a similar claim
is valid if we replace $Q_0S^0$ by $QS^n$ with $n>0$ in the above
theorem.

\subsection{Completing the proof of the Main Theorem}

Our main theorem is stated in terms of primitive classes in
$H_{2i+1}Q_0S^0$. Recall that any spherical class in $H_*Q_0S^0$ may
be written as a linear combination of terms of the form
$Q^Ip_{2i+1}$ with $I$, and not necessarily $(I,2i+1)$, being
admissible. The class $p_{2i+1}\in H_{2i+1}Q_0S^0$ is the unique
primitive such that
$$p_{2i+1}=x_{2i+1}+D_{2i+1}$$
with $D_{2i+1}$ being is sum of decomposable terms. This then makes
the following observation evident.
\begin{lmm}
If $(I,2n+1)$ is admissible, then modulo decomposable terms
$$Q^Ip_{2n+1}=Q^Ix_{2n+1}.$$
\end{lmm}

Notice that a spherical class is $A$-annihilated and primitive.
Combining this with Remark 3.11 enables us to have the following
observation.
\begin{lmm}
Suppose $\xi_0\in H_*Q_0S^0$ is an $A$-annihilated primitive class
with $\sigma_*\xi_0\neq 0$. Then
$$\xi_0=\sum Q^Ix_{2i+1}$$
modulo decomposable terms, where $(I,2i+1)$ runs over certain
admissible sequences of positive excess.
\end{lmm}

\begin{proof}
The fact that $\sigma_*\xi_0\neq 0$ implies that modulo decomposable
terms
$$\xi_0=\sum Q^Ix_n$$
where $(I,n)$ is admissible with $\ex(Q^Ix_n)>0$. The fact that
$\xi_0$ is an indecomposable primitive implies that indecomposable
part of $\xi_0$ belongs to the kernel of the square root map
$r:H_*Q_0S^0\to H_*Q_0S^0$. Notice that if we have two distinct
admissible sequences $(J,j)$ and $(K,k)$ with only even entries,
then $rQ^Jx_j\neq rQ^Kx_k$. Hence the indecomposable part of $\xi_0$
belongs to the kernel of $r$ if and only if every $Q^Ix_n$ belong to
the kernel. We show that assuming $n\neq 2i+1$ leads to a
contradiction.\\
Assume that $n$ is even. Since $Q^Ix_n$ belong to $\ker r$, then
$I=(i_1,\ldots, i_t)$ must have at least one odd entry. Let
$s_0=\max(s:1\leqslant s\leqslant t,i_s\textrm{ is odd})$. Then
$i_{s_0+1}$ is even. Notice that if $s_0=t$, then we have $x_n$ next
to it which is even. In this case one applies $Sq^{2^{s_0}}_*$ to
$\xi_0$. Notice that according to Theorem 1 a class $Q^Ix_n$ with
$(I,n)$ having at least one even entry is not $A$-annihilated.
Moreover, according to Note 4, all terms of the form $Q^Ix_n$ with
$\ex(Q^Ix_n)>0$ are separated under the action of this operation
from each other. Moreover, notice that
$$\ex(Sq^{2^{s_0}}_*Q^Ix_n)=\ex(Q^Ix_n)>0,$$
which implies that the outcome is not a decomposable, and hence is
separated from any other decomposable term. This implies that
$Sq^{2^{s_0}}_*\xi_0\neq 0$ which contradicts the fact that $\xi_0$
must be $A$-annihilated. Hence $n$ must be odd. This implies that
modulo decomposable terms
$$\xi=\sum Q^Ix_{2i+1},$$
with $(I,2i+1)$ admissible.
\end{proof}

The previous lemma together with the above observation, leads us to
a better and more clear expression for a potential spherical classes
in $H_*Q_0S^0$. We have the following.

\begin{crl}
Let $\xi_0\in H_*Q_0S^0$ be $A$-annihilated primitive class with
$\sigma_*\xi_0\neq 0$. Then
$$\xi_0=\sum Q^Ip_{2i+1}$$
with $(I,2i+1)$ admissible modulo decomposable terms. If $\xi_0$ is
odd dimensional, then the decomposable part is trivial. If $\xi_0$
is even dimensional, then the decomposable part is either trivial or
square of a primitive.\\
Moreover, assume that $l(I)>1$ for every $I$ involved in the above
expression for $\xi_0$. Then $I$ will have only odd entries.
\end{crl}

\begin{proof}
Notice that $\xi_0=\sum Q^Ix_{2i+1}$ modulo decomposable terms.
Previous lemma allows us to replace $Q^Ix_{2i+1}$ with $Q^Ip_{2i+1}$
modulo decomposable terms. Therefore $\xi_0=\sum Q^Ip_{2i+1}$ modulo
decomposable terms. However this decomposable part is primitive,
hence it must be square. If $\xi_0$ is an odd dimensional class,
then the decomposable part is trivial. If $\xi_0$ is even
dimensional then it is either square or trivial.\\
If $\xi_0$ is even dimensional, then we may write
$$\xi_0=\sum Q^Ip_{2i+1}+P^2$$
with $(I,2i+1)$ admissible, and $P$ a primitive class. The fact that
decomposable terms die under suspension, together with the Lemma
3.13 show that
$$\sigma_*Q^Ip_{2i+1}=\sigma_* Q^Ix_{2i+1}=Q^IQ^{2i+1}g_1\in
H_*QS^1.$$ Hence, we have
$$\sigma_*\xi_0=\sum Q^IQ^{2i+1}g_1$$
being an odd dimensional $A$-annihilated class. Corollary $3$, then
implies that each term $Q^IQ^{2i+1}g_1$ must be $A$-annihilated
which in particular means that $I$ has only odd entries, and
$(I,2i+1)$ satisfies condition $3$ of Theorem 1.\\
If $\xi_0$ is odd dimensional, then
$$\xi_0=\sum Q^Ip_{2i+1}$$
with $(I,2i+1)$ being admissible. If we have $I=(i_1,\ldots,i_r)$,
with $i_s$ being even, then applying $Sq^{2^{s-1}}_*$ will show that
$\xi_0$ is not $A$-annihilated which is a contradiction. This then
shows that $I$ must have only odd entries. This completes the proof.
\end{proof}

To complete the proof, we need to translate the above results and
express our results in terms of $Q^Ip'_{2i+1}$ whereas the above
results are expressed in terms of primitive classes $Q^Ip_{2i+1}$.
The main reason for this, is the difference that $p_{2i+1}$ and
$p'_{2i+1}$ show under the action of the Steenrod algebra. The
following remark explains this more.
\begin{rmk}
The action of the Steenrod algebra on $H_*SO$ is given by
$$Sq^k_*p^{SO}_{2n+1}={2n+1-k\choose k}p^{SO}_{2n+1-k}.$$
This makes it easy to see that the primitive classes $p'_{2n+1}$
behave similar to $x_{2n+1}$ under the action of the Steenrod
algebra, i.e.
$$Sq^k_*p'_{2n+1}={2n+1-k\choose k}p'_{2n+1-k}.$$
Note that the above action is trivial when $k$ is odd.\\
We refer the reader to \cite[Lemma 5.2]{15} to see that
$$Sq^k_*p_n={n-k-1\choose k}p_{n-k}.$$
Notice that mod $2$, we have
$${2i-2k\choose 2k}={2i-2k+1\choose 2k}.$$
This implies that $p_{2n+1}$, and $p'_{2n+1}$ behave in the same way
under the operations $Sq^{2k}_*$. However, the primitive classes
$p'_{2n+1}$ are annihilated under the operations $Sq^{2k+1}_*$
whereas the primitive classes $p_{2n+1}$ have chance to survive
under these operation, e.g. $Sq^1_*p_3=x_1^2$.\\
Recall that the action of $Sq^k_*$ on $x_{2i+1}$ is given by
$$Sq^k_*x_{2i+1}={2n+1-k\choose k}x_{2n+1-k}.$$
This shows that the primitive classes $p'_{2i+1}$ behaves like
$x_{2i+1}$ under the action of the Steenrod operation. On the other
hand, the class $Q^Ix_{2i+1}$ with $(I,2i+1)$ admissible behaves
like $Q^IQ^{2i+1}$ under the action of the Steenrod algebra. This
then shows that if $(I,2i+1)$ is admissible and $l(I)>0$, then
$Q^Ip'_{2i+1}$ behaves like $Q^IQ^{2i+1}$ under the action of the
Steenrod algebra. Notice that this claim is not true about the
classes $Q^Ip_{2i+1}$.
\end{rmk}

The above remark combined with Note 4 implies the following.
\begin{thm}
Suppose $\sum Q^Ip'_{2i+1}$ is an $A$-annihilated sum, with
$l(I)>0$. Then each term $Q^Ip'_{2i+1}$ must be $A$-annihilated.
\end{thm}

Next, we compare the behavior of primitive classes $Q^Ip'_{2i+1}$
and $Q^Ip_{2i+1}$ under the homology suspension.
\begin{rmk}
We like to draw the reader's attention to the behavior of
$Q^Ip_{2n+1}$, and $Q^Ip_{2n+1}'$ under the homology suspension.
First let $I=\phi$. Recall that modulo decomposable terms,
$$\begin{array}{lll}
p_{2n+1} & = & x_{2n+1},\\
p_{2n+1}'& = & x_{2n+1}+Q^{2n}x_1.
\end{array}$$
We then obtain,
$$\begin{array}{lll}
\sigma_*p_{2n+1} & = & Q^{2n+1}g_1,\\
\sigma_*p_{2n+1}'& = & Q^{2n+1}g_1+(Q^ng_1)^2.
\end{array}$$
Now suppose $I=(i_1,\ldots,i_r)$ is a sequence with $i_r$ odd such
that $(I,2n+1)$ is admissible. We then have
$$\sigma_*Q^IQ^{2n}x_1=Q^IQ^{2n}g_1^2=Q^I(Q^ng_1)^2=0.$$
In fact we don't need to restrict to $i_r$, similar statement
holds if we assume only $I$ has at least one odd entry.\\
Notice that $Q^IQ^{2n}x_1$ is a primitive class, which can be
written in terms of $Q^Jx_j$ modulo decomposable terms, where
$(J,j)$ is admissible. Any class $Q^Jx_j$ with $(J,j)$ dies under
suspension, if and only if $\ex(Q^Jx_j)=0$, i.e. $Q^Jx_j$ is
decomposable. Hence $Q^IQ^{2n}x_1$ is a decomposable primitive, and
hence a square term. We then observe that if we choose $I$ to be
even dimensional, then $Q^IQ^{2n}x_1$ is odd dimensional which makes
it impossible to be a square, hence $Q^IQ^{2n}x_1=0$. In this case
we have
$$\begin{array}{lllll}
\sigma_*Q^Ip_{2n+1}'&=&Q^IQ^{2n+1}g_1&=&\sigma_*Q^Ip_{2n+1},
\end{array}$$
as well as
$$Q^Ip_{2n+1}=Q^Ip_{2n+1}.$$
\end{rmk}

Now, we restrict our attention to the classes $Q^Ip_{2i+1}$ with
$l(I)>0$. The result reads as following.

\begin{lmm}
Suppose $\xi_0\in H_*Q_0S^0$ is an odd dimensional $A$-annihilated
primitive class. Then
$$\xi_0=\sum Q^Ip'_{2i+1},$$
with $(I,2i+1)$ admissible, and each $Q^Ip_{2i+1}$ being
$A$-annihilated.
\end{lmm}

The above theorem is just a combination of Corollary 3.15, Theorem
3.17, and Remark 3.18.\\

Our main theorem now follows from combining Theorem 3.10, Lemma
3.11, Corollary 3.15 and Lemma 3.19.

\end{document}